\documentclass[a4paper,11pt,fleqn]{article}
\usepackage{amsmath}
\usepackage{amssymb}
\usepackage{theorem}
\usepackage{euscript}
\usepackage{srcltx}
\usepackage[dvipsnames]{xcolor}
\usepackage[bitstream-charter]{mathdesign}
\usepackage[T1]{fontenc}
\usepackage{epsfig}
\usepackage[usenames,dvipsnames]{pstricks}
\usepackage{pstricks-add}
\usepackage{pst-grad}
\usepackage{pst-plot} 
\usepackage{pst-node} 
\usepackage{pst-eucl} 
\newcommand{\scal}[2]{{\left\langle{{#1}\mid{#2}}\right\rangle}}
\newcommand{\conv}[1]{\ensuremath{\text{\rm conv}}(#1)}
\newcommand{\maximize}[2]{\ensuremath{\underset{\substack{{#1}}}%
{\text{\rm maximize}}\;\;#2 }}

\newcommand{\Frac}[2]{\displaystyle{\frac{#1}{#2}}} 
\newcommand{\RX}{\ensuremath{\left]-\infty,+\infty\right]}}
\renewcommand{\leq}{\ensuremath{\leqslant}}
\renewcommand{\geq}{\ensuremath{\geqslant}}
\renewcommand{\le}{\ensuremath{\leqslant}}

\newcommand{\inte}{\ensuremath{\text{\rm int}}}
\newcommand{\card}{\ensuremath{\text{\rm card}\,}}
\newcommand{\dom}{\ensuremath{\text{\rm dom}\,}}
\newcommand{\XX}{\ensuremath{\mathcal{X}}}
\newcommand{\menge}[2]{\big\{{#1}~\big |~{#2}\big\}} 
\newcommand{\Mmenge}[2]{\bigg\{{#1}~\Big |~{#2}\bigg\}} 
\newcommand{\Menge}[2]{\Big\{{#1}~\big |~{#2}\Big\}} 
\newtheorem{theorem}{Theorem}[section]
\newtheorem{lemma}[theorem]{Lemma}
\newtheorem{corollary}[theorem]{Corollary}

\theoremstyle{plain}{\theorembodyfont{\rmfamily}%
\newtheorem{conjecture}[theorem]{Conjecture}}
\theoremstyle{plain}{\theorembodyfont{\rmfamily}%
\newtheorem{assumption}[theorem]{Assumption}}
\theoremstyle{plain}{\theorembodyfont{\rmfamily}%
}
\theoremstyle{plain}{\theorembodyfont{\rmfamily}%
\newtheorem{example}[theorem]{Example}}
\theoremstyle{plain}{\theorembodyfont{\rmfamily}%
\newtheorem{remark}[theorem]{Remark}}
\theoremstyle{plain}{\theorembodyfont{\rmfamily}%
\newtheorem{definition}[theorem]{Definition}}
\theoremstyle{plain}{\theorembodyfont{\rmfamily}%
}

\numberwithin{equation}{section}
\title{\sffamily Kolmogorov $n$-Widths of Function Classes 
Induced by a\\ Non-Degenerate Differential 
Operator: A Convex Duality Approach\footnote{Contact author: 
P. L. Combettes, {\ttfamily{plc@ljll.math.upmc.fr}},
phone: +33 1 4427 6319, fax: +33 1 4427 7200.}}
\author{Patrick L. Combettes$^1$ and Dinh D\~ung$^2$\\[5mm]
\small $\!^1$Sorbonne Universit\'es -- UPMC Univ. Paris 06\\
\small UMR 7598, Laboratoire Jacques-Louis Lions\\
\small F-75005 Paris, France\\
\small \ttfamily{plc@ljll.math.upmc.fr}\\[4mm]
\small $\!^2$Information Technology Institute\\
\small Vietnam National University\\
\small Hanoi, Vietnam\\ 
\small {\ttfamily{dinhzung@gmail.com}}\\[4mm]
}
\date{~}
\newcommand{\pinf}{\ensuremath{{+\infty}}}
\newcommand{\RP}{\ensuremath{{\mathbb R}_+}}
\newcommand{\RR}{\ensuremath{{\mathbb R}}}
\newcommand{\TT}{\ensuremath{{\mathbb T}}}
\newcommand{\CC}{\ensuremath{{\mathbb C}}}
\newcommand{\RPP}{\ensuremath{{\mathbb R}_{++}}}
\newcommand{\exi}{\ensuremath{\exists\,}}
\newcommand{\ZZ}{\ensuremath{\mathbb Z}}
\newcommand{\Ss}{\ensuremath{\mathcal S}}
\newcommand{\ZZd}{\ensuremath{{\mathbb Z}^d}}

\newcommand{\ZZdp}{\ensuremath{{\mathbb N}^d}}
\newcommand{\NN}{\ensuremath{\mathbb N}}
\newcommand{\NNd}{\ensuremath{\mathbb N}^d}
\newcommand{\RRd}{\ensuremath{\mathbb R}^d}
\newcommand{\TTd}{\ensuremath{\mathbb T}^d}
\newcommand{\RRdp}{\ensuremath{\mathbb R}_+^d}
\newcommand{\WP}{\ensuremath{{W^{[P]}_2}}}
\newcommand{\UP}{\ensuremath{{U^{[P]}_2}}}
\begin{document}
\maketitle

\begin{abstract}
Let $P(D)$ be the differential operator induced by a polynomial 
$P$, and let $\UP$ be the class of multivariate periodic functions 
$f$ such that $\|P(D)(f)\|_2\leq 1$. The problem of
computing the asymptotic order of the Kolmogorov $n$-width 
$d_n(\UP,L_2)$ in the general case when $\UP$ is compactly 
embedded into $L_2$ has been open for a long time. In the 
present paper, we use convex analytical tools to solve it in 
the case when $P(D)$ is non-degenerate. 
\end{abstract}



\newpage
\section{Introduction}

The aim of the present paper is to study Kolmogorov $n$-widths of
classes of multivariate periodic functions induced by a differential
operator. In order to describe the exact setting of the problem 
let us introduce some notation. 

We first recall the notion of Kolmogorov $n$-widths 
\cite{Ko36,Pink85}. Let $\mathcal{X}$ be a normed space, 
let $F$ be a nonempty subset of $\mathcal{X}$ such that $F=-F$, 
and let ${\mathcal G}_n$ be the class of all vector subspaces of 
$\XX$ of dimension at most $n$. The 
Kolmogorov $n$-width of $F$ in $\XX$ is 
\begin{equation}
\label{e:1}
d_n(F,\mathcal{X})=\inf_{G\in{\mathcal G}_n}\,\sup_{f\in F}\,
\inf_{g\in G}\,\|f-g\|_\XX.
\end{equation}
This notion quantifies the error of the best approximation to the 
elements of $F$ by elements in a vector subspace of 
$\XX$ of dimension at most $n$ \cite{Pink85,Ti60,Ti76}. 

In computational mathematics, the so-called $\varepsilon$-dimension
$n_\varepsilon(F,\XX)$ is used to quantify the 
computational complexity. It is defined by
\begin{equation}
\label{e:2}
n_\varepsilon(F,\XX) = 
\inf\Mmenge{n\in\NN}{(\exi G\in\mathcal{G}_n)\; \sup_{f \in F} \ 
\inf_{g \in G} \|f-g\|_\XX \le \varepsilon}.
\end{equation}
This approximation characteristic is the inverse of $d_n(F,\XX)$ 
in the sense that the quantity $n_\varepsilon(F,\XX)$ is the 
smallest integer
$n_\varepsilon$ such that the approximation of $F$ by a suitably
chosen approximant $n_\varepsilon$-dimensional subspace $G$ in 
$\XX$ gives an approximation error less than $\varepsilon$. 
Recently, there has been strong interest in applications of 
Kolmogorov $n$-widths, and its dual Gelfand $n$-widths, to 
compressive sensing \cite{BDDW08,Do06,FPRU10,Pink11}. 
Kolmogorov $n$-widths and $\varepsilon$-dimensions of classes 
of functions with mixed smoothness have also 
been employed in recent high-dimensional approximation 
studies \cite{CD13,DU13}. 

We consider functions on $\RRd$ which are $2\pi$-periodic in
each variable as functions defined on $\TTd=[-\pi,\pi]^d$. 
Denote by $L_2(\TTd)$ the Hilbert space of square-integrable
functions on $\TTd$ equipped with the standard scalar product, i.e., 
\begin{equation}
(\forall f\in L_2(\TTd))(\forall g\in L_2(\TTd))\quad
\scal{f}{g}=\frac{1}{(2\pi)^d}\int_{\TTd}f(x)\overline{g(x)} dx,
\end{equation}
and by $\Ss'(\TTd)$ the space of distributions on $\TTd$. 
The norm of $f\in L_2(\TTd)$ is $\|f\|_2=\sqrt{\scal{f}{f}}$ and, 
given $k\in\ZZ^d$, the $k$th Fourier coefficient of 
$f\in L_2(\TTd)$ is $\hat{f}(k)=\scal{f}{e^{i\scal{k}{\cdot}}}$. 
Every $f\in\Ss'(\TTd)$ can be identified with the formal 
Fourier series 
\begin{equation}
f=\sum_{k\in\ZZ^d}\hat{f}(k)e^{i\scal{k}{\cdot}},
\end{equation}
where the sequence $(\hat{f}(k))_{k\in\ZZ^d}$ is a tempered 
sequence \cite{Sch66,Ti76}. By Parseval's identity, $L_2(\TTd)$ is 
the subset of $\Ss'(\TTd)$ of all distributions $f$ for which
\begin{equation}
\sum_{k\in\ZZd}|{\hat f}(k)|^2 <\pinf.
\end{equation} 
Let $\alpha=(\alpha_1,\ldots,\alpha_d)\in\ZZdp$ and let
$f\in\Ss'(\TTd)$. We set 
\begin{equation}
\ZZ^d_0(\alpha)=\menge{(k_1,\ldots,k_d)\in\ZZ^d}
{(\forall j\in\{1,\ldots,d\})\;\;\alpha_j\neq 0\;\;
\Rightarrow\;\;k_j\neq 0}.
\end{equation}
As usual, we set $|\alpha|=\sum_{j=1}^d\alpha_j$ and, given 
$z=(z_1,\ldots,z_d)\in\CC^d$, we set
$z^\alpha=\prod_{j=1}^dz_j^{\alpha_j}$.
The $\alpha$th derivative of $f\in\Ss'(\TTd)$ is the distribution 
$f^{(\alpha)}\in\Ss'(\TTd)$ given through the identification
\begin{equation} 
\label{def[f^{(r)}]}
f^{(\alpha)}=
\sum_{k\in\ZZ^d_0(\alpha)} (ik)^\alpha \hat{f}(k) 
e^{i\scal{k}{\cdot}}.
\end{equation}
The differential operator $D^\alpha$ on $\Ss'(\TTd)$ is defined by 
$D^\alpha\colon f\mapsto (-i)^{|\alpha|} f^{(\alpha)}$. Now let 
$A\subset\ZZdp$ be a nonempty finite set, let 
$(c_\alpha)_{\alpha\in A}$ be 
nonzero real numbers, and define a polynomial by
\begin{equation} 
P\colon x\mapsto\sum_{\alpha\in A} c_\alpha x^\alpha.
\end{equation}
The differential operator $P(D)$ on $\Ss'(\TTd)$ induced by 
$P$ is 
\begin{equation} 
\label{def[P(D)]}
P(D)=\sum_{\alpha\in A} c_\alpha D^\alpha. 
\end{equation}
Set 
\begin{equation} 
\WP=\menge{f\in\Ss'(\TTd)}{P(D)(f)\in L_2(\TTd)},
\end{equation}
denote the seminorm of $f\in\WP$ by 
\begin{equation}
\|f\|_{\WP}=\|P(D)(f)\|_2,
\end{equation}
and let 
\begin{equation} 
\label{e:UP}
\UP=\Menge{f\in\WP}{\|f\|_{\WP}\leq 1}.
\end{equation}
The problem of computing asymptotic orders of 
$d_n(\UP,L_2(\TTd))$ in the general case when $\WP$ is compactly 
embedded into $L_2(\TTd)$ has been open for a long time; see, e.g., 
\cite[Chapter~III]{Te93} for details.
Our main contribution is to solve it for a non-degenerate
differential operator $P(D)$ (see Definition~\ref{d:nondegenerate}). 
Using convex-analytical tool, we establish the asymptotic order
\begin{equation} 
\label{e:2013-09-17}
d_n\big(\UP,L_2(\TTd)\big)\asymp n^{-\varrho}(\log n)^{\nu\varrho},
\end{equation}
where $\varrho$ and $\nu$ depend only on $P$.

The first exact values of $n$-widths of univariate Sobolev
classes were obtained  by Kolmogorov \cite{Ko36} (see also
\cite[pp.~186--189]{Ko85}). The problem of computing the
asymptotic order of $d_n(\UP, L_2(\TTd))$ is directly related to 
hyperbolic crosses trigonometric approximations and to $n$-widths 
of classes multivariate periodic functions with a bounded mixed 
smoothness. This line of work was initiated by Babenko in
\cite{Ba60a,Ba60}. In particular, the asymptotic orders of
$n$-widths in $L_2(\TTd)$ of these classes were established in 
\cite{Ba60a}. Further work on asymptotic orders and hyperbolic cross
approximation can be found in \cite{DD86,DD97,Te93} and recent
developments in \cite{Kush12,Schm04,Sick09,Wang12}.
In \cite{DD84}, the strong asymptotic order 
of $d_n(U^A_2,L_2(\TTd))$ was computed in the case when $U^A_2$ is
the closed unit ball of the space $W^A_2$ of functions with several 
bounded mixed derivatives (see Subsection~\ref{W^A_2} for a precise
definition). 

The remainder of the paper is organized as follows. 
In Section~\ref{Preliminary}, we provide as auxiliary results 
Jackson-type and Bernstein-type inequalities for trigonometric 
approximations of functions from $\WP$. We also characterize
the compactness of $\UP$ in $L_2(\TTd)$ and the non-degenerateness 
of $P(D)$. 
In Section~\ref{Widths}, we present the main result of the
paper, namely the asymptotic order of $d_n\big(\UP, L_2(\TTd)\big)$
in the case when $P(D)$ is non-degenerate.
In Section~\ref{Norms equivalences}, we derive norm equivalences
relative to $\|\cdot\|_{\WP}$ and, based on them, we
provide examples of $n$-widths $d_n(\UP,L_2(\TTd))$ 
for non-degenerate differential operators.

\section{Preliminaries} 
\label{Preliminary}

\subsection{Notation, standing assumption, and definitions}

We set $\NN=\{0,1,\ldots,\}$, $\NN^*=\{1,2,\ldots,\}$, 
$\RP=\left[0,\pinf\right[$, and $\RPP=\left]0,\pinf\right[$. 
Let $\Theta$ be an abstract set, and let $\Phi$ and $\Psi$ be
functions from $\Theta$ to $\RR$. Then we write
\begin{equation}
(\forall \theta\in\Theta)\quad\Phi(\theta)\asymp\Psi(\theta) 
\end{equation}
if there exist $\gamma_1\in\RPP$ and $\gamma_2\in\RPP$ such that
$(\forall \theta\in\Theta)$ $\gamma_1\Phi(\theta)\leq\Psi(\theta)\leq
\gamma_2\Phi(\theta)$.
For every $j\in\{1,\ldots,d\}$, $u^j$ denotes the $j$ standard unit
vector of $\RR^d$  and 
\begin{equation}
\label{e:hanoi2014-05-23h}
\mathcal{R}^j=\menge{\lambda u^j}{\lambda\in\RPP}
\end{equation}
the $j$th standard strict ray.

\begin{definition}
\label{d:hanoi2014-05-21}
Let $B$ be a nonempty finite subset of $\NN^d$. The convex hull 
$\conv{B}$ of $B$ is the polyhedron spanned by $B$,
\begin{equation}
\label{e:hanoi2014-05-21b}
\Delta(B)=\Menge{\alpha\in B}{\menge{\lambda\alpha}{\lambda\in
\left[1,\pinf\right[}\cap\conv{B}=\{\alpha\}},
\end{equation}
and $\vartheta(B)$ is the set of vertices of 
$\conv{\Delta(B)}$. In addition,
\begin{equation}
\label{e:Omega}
(\forall t\in\RP)\quad
\Omega_B(t)=\Mmenge{k\in\ZZdp}{\max_{\alpha\in B} k^\alpha\leq t}.
\end{equation}
\end{definition}

Throughout the paper, the convention $0^0$ is adopted and the
following standing assumption is made.

\begin{assumption}
\label{a:standing}
$A$ is a nonempty finite subset of $\ZZdp$ and 
$(c_\alpha)_{\alpha\in A}$ are nonzero real numbers.
We set 
\begin{equation} 
\label{def[P]}
P\colon x\mapsto\sum_{\alpha\in A} c_\alpha x^\alpha
\quad\text{and}\quad\tau=\inf_{k\in\ZZd}|P(k)|.
\end{equation}
Moreover, for every $t\in\RP$, we set
\begin{equation}
K(t)=\menge{k\in\ZZd}{|P(k)|\leq t}
\quad\text{and}\quad
V(t)=\Mmenge{f\in\Ss'(\TTd)}{
f=\sum_{k \in K(t)}{\hat f}(k) e^{i\scal{k}{\cdot}}}.
\end{equation}
\end{assumption}

\begin{remark}
If $0\in A$, then $0\in\vartheta(A)$ and
$\Delta(\conv{A})=\Delta(A)$, so that
$\vartheta(\conv{A})=\vartheta(A)$.
Now suppose that $t\in\left]\tau,\pinf\right[$. Then 
$K(t)\neq\varnothing$ and $\dim V(t)=\card K(t)$, 
where $\card K(t)$ denotes the cardinality of $K(t)$. 
In addition, if $\card K(t)<\pinf$,
then $V(t)$ is the space of trigonometric polynomials with 
frequencies in $K(t)$. 
\end{remark}

\begin{definition}
\label{d:nondegenerate}
The \emph{Newton diagram} of $P$ is $\Delta(A)$ and the
\emph{Newton polyhedron} of $P$ is $\conv{A}$.
The intersection of $\conv{A}$ with a supporting hyperplane 
of $\conv{A}$ is a \emph{face} of $\conv{A}$; 
$\Sigma(A)$ is the set of intersections of $A$ with a face of 
$\conv{A}$.
The differential operator $P(D)$ is \emph{non-degenerate} if
$P$ and, for every $\sigma\in\Sigma(A)$, 
$P_\sigma\colon\RR^d\to\RR\colon x\mapsto
\sum_{\alpha\in\sigma}c_\alpha x^\alpha$ do not vanish outside 
the coordinate planes of $\RR^d$, i.e.,  
\begin{equation} 
\label{condition[|P(x)|>]}
\big(\forall x\in\RRd\big)\quad\Bigg(
\prod_{j=1}^d x_j \neq 0\quad\Rightarrow\quad
\big(\forall\sigma\in\Sigma(A)\big)\quad P(x)P_\sigma(x)\neq
0\Bigg).
\end{equation}
\end{definition}

\begin{remark}
\label{r:hanoi2014-04-23}
Suppose that $P$ is non-degenerate and let $\alpha\in\vartheta(A)$. 
Then it follows from \eqref{condition[|P(x)|>]} that all the 
components of $\alpha$ are even.
\end{remark}

\subsection{Trigonometric approximations}

We first prove a Jackson-type inequality.

\begin{lemma} 
\label{l:10}
Let $t\in\RPP$ and define a linear operator 
$S_t\colon\Ss'(\TTd)\to\Ss'(\TTd)$ by
\begin{equation}
\label{e:2013-09-20a}
\big(\forall f\in \Ss'(\TTd)\big)\quad 
S_t(f)=\sum_{k \in K(t)} {\hat f}(k) e^{i\scal{k}{\cdot}}.
\end{equation}
Let $f\in\WP$ and suppose that $t>\tau$. 
Then the distribution $f-S_t(f)$ represents a function in
$L_2(\TTd)$ and 
\begin{equation} 
\label{ineq[|f-S_t(f)|]}
\|f-S_t(f)\|_2\leq t^{-1}\|f\|_{\WP}.
\end{equation}
\end{lemma}
\begin{proof}
Set $g=f-S_t(f)$. Then $g\in\Ss'(\TTd)$. On the other hand,
Parseval's identity yields
\begin{equation} 
\label{e:parseval}
\|f\|_{\WP}^2=\sum_{k\in\ZZd} |P(k)|^2 |{\hat f}(k)|^2.
\end{equation} 
Hence, 
\begin{align}
\sum_{k\in\ZZd} |{\hat g}(k)|^2
&=\sum_{k\in\ZZd\setminus K(t)} |{\hat f}(k)|^2
\nonumber\\[1.5ex]
&\leq\bigg(\sup_{k\in\ZZd\setminus K(t)} |P(k)|^{-2}\bigg) 
\sum_{k\in\ZZd\setminus K(t)} |P(k)|^2 |{\hat f} (k)|^2 
\nonumber\\[1.5ex]
&\leq t^{-2}\|f\|_{\WP}^2,
\end{align}
which means that $f-S_t(f)$ represents a function in $L_2(\TTd)$ for 
which \eqref{ineq[|f-S_t(f)|]} holds.
\end{proof}

\begin{corollary} 
\label{corollary[|f-S_t(f)|<]}
Let $t\in\left]\tau,\pinf\right[$. Then 
\begin{equation}
\sup_{f\in\UP} \ 
\inf_{\substack{g \in V(t)\\ f-g \in L_2(\TTd)}}  \|f-g\|_2
\leq t^{-1}.
\end{equation}
\end{corollary}

Next, we prove a Bernstein-type inequality. 

\begin{lemma} 
\label{l:B} 
Let $t\in\left]\tau,\pinf\right[$ and let 
$f\in V(t)\cap L_2(\TTd)$. Then
\begin{equation}
\label{e:fn439n}
\|f\|_{\WP}\leq t\|f\|_2.
\end{equation}
\end{lemma}
\begin{proof}
By \eqref{e:parseval}, we have 
\begin{align}
\|f\|_{\WP}^2
=\sum_{k \in K(t)} |P(k)|^2 |{\hat f} (k)|^2 
\leq\bigg(\sup_{k \in K(t)}|P(k)|^2\bigg)  
\sum_{k \in K(t)}|{\hat f} (k)|^2 
\leq t^2\|f\|_2^2,
\end{align}
which establishes \eqref{e:fn439n}.
\end{proof}

\subsection{Compactness and non-degenerateness}

We start with a characterization of the compactness of the 
unit ball defined in \eqref{e:UP}.

\begin{lemma} 
\label{l:7} 
The set $\UP$ is a compact subset of $L_2(\TTd)$ if and only if 
the following hold:
\begin{enumerate}
\item
\label{l:1i}
For every $t\in\left]\tau,\pinf\right[$, $K(t)$ is finite. 
\item 
\label{l:1ii}
$\tau>0$.
\end{enumerate}
\end{lemma}
\begin{proof}
To prove sufficiency, suppose that \ref{l:1i} and \ref{l:1ii} hold,
and fix $t\in\left]\tau,\pinf\right[$.
By \ref{l:1i}, $V(t)$ is a set of trigonometric 
polynomials and, consequently, a subset of $L_2(\TTd)$. 
In particular, using the notation \eqref{e:2013-09-20a},
$(\forall f\in\Ss'(\TTd))$ $S_t(f)\in L_2(\TTd)$.
Hence, by Lemma~\ref{l:10}, 
\begin{equation}
\Big(\forall f\in\WP\Big)\quad f=(f-S_t(f))+S_t(f)\in L_2(\TTd). 
\end{equation}
Thus, $\WP\subset L_2(\TTd)$. On the other hand, 
\eqref{e:parseval} implies that 
$\UP$ is a closed subset of $L_2(\TTd)$. 
Therefore, $\UP$ is compact in $L_2(\TTd)$ if,
for every $\varepsilon\in\RPP$, it has a finite 
$\varepsilon$-net in $L_2(\TTd)$ or, equivalently, 
if the following following two conditions are satisfied:
\begin{enumerate}
\setcounter{enumi}{2}
\item 
\label{l:1iii}
For every $\varepsilon\in\RPP$, there exists a finite-dimensional 
vector subspace $G_\varepsilon$ of $L_2(\TTd)$ such that
\begin{equation}
\sup_{f \in\UP} \, \inf_{g\in G_\varepsilon}\|f-g\|_2\leq 
\varepsilon.
\end{equation}
\item
\label{l:1iv}
$\UP$ is bounded in $L_2(\TTd)$.
\end{enumerate}
It follows from \eqref{e:parseval} that 
\ref{l:1ii}$\Leftrightarrow$\ref{l:1iv}. On the other hand, 
since $\dim V(t)=\card K(t)$, 
Corollary~\ref{corollary[|f-S_t(f)|<]} 
yields \ref{l:1i}$\Rightarrow$\ref{l:1iii}.
To prove necessity, suppose that \ref{l:1i} does not hold. 
Then $\dim V(\tilde{t})=\card K(\tilde{t})=\pinf$ for some 
$\tilde{t}\in\RPP$. By Lemma~\ref{l:B}, 
$\widetilde{U}=\menge{f\in V(\tilde{t})\cap L_2(\TTd)}{\|f\|_2
\leq 1/\tilde{t}}$ is a subset of $\UP$ which is not compact 
in $L_2(\TTd)$. If \ref{l:1ii} does not hold, then
$\UP\cap L_2(\TTd)$ is unbounded and, consequently, 
not compact in $L_2(\TTd)$. 
\end{proof}

The following lemma characterizes the non-degenerateness of $P(D)$. 

\begin{lemma}
\label{l:non-deg} 
$P(D)$ is non-degenerate if and only if
\begin{equation}
\label{e:2013-09-16}
(\exi\gamma\in\RPP)(\forall x\in\RRd)\quad |P(x)|\geq\gamma 
\max_{\alpha\in\vartheta(A)} |x^\alpha|.
\end{equation}
\end{lemma}
\begin{proof}
As proved in \cite{Gi74,Mi67}, $P(D)$ is non-degenerate if and 
only if
\begin{equation}
\label{e:2013-09-17b}
(\exi \gamma_1\in\RPP)(\forall x\in\RRd)\quad |P(x)|\geq 
\gamma_1\sum_{\alpha\in\vartheta(A)}|x^\alpha|.
\end{equation}
Hence, since there exist $\gamma_2\in\RPP$ and $\gamma_3\in\RPP$ 
such that 
\begin{equation}
(\forall x\in\RR^d)\quad\gamma_2\max_{\alpha\in\vartheta(A)}
|x^\alpha|\leq\sum_{\alpha\in\vartheta(A)}|x^\alpha|\leq
\gamma_3\max_{\alpha\in\vartheta(A)}|x^\alpha|, 
\end{equation}
the proof is complete.
\end{proof}

\begin{lemma} 
\label{l:6} 
Let $B$ be a nonempty finite subset of $\NN^d$ and
let $t\in\RP$. Then 
\begin{equation}
\Omega_B(t)=\Mmenge{k\in\ZZdp}{\max_{\alpha\in B} k^\alpha\leq t}
\end{equation}
is finite if and only if
\begin{equation}
\label{e:hanoi2014-05-22}
(\forall j\in\{1,\ldots,d\})\quad
B\cap\mathcal{R}^j\neq\varnothing.
\end{equation}
\end{lemma}
\begin{proof}  
If \eqref{e:hanoi2014-05-22} holds, then 
$(\forall j\in\{1,\ldots,d\})(\exi a_j\in\RPP)$
$a_ju^j\in B\cap\mathcal{R}^j$. Hence, 
\eqref{e:Omega} implies that $\Omega_B(t)\subset\bigcap_{j=1}^d
\menge{k\in\NN^d}{k_j\leq t^{1/a_j}}$ and, therefore, 
$\Omega_B(t)$ is bounded. Conversely, if \eqref{e:hanoi2014-05-22} 
does not hold, then there exists $j\in\{1,\ldots,d\}$ such that 
$\menge{mu^j}{m\in\NN}\subset\Omega_B(t)$, which shows that
$\Omega_B(t)$ is unbounded.
\end{proof}

\begin{theorem} 
\label{theorem[compactness(2)]} 
Suppose that $P(D)$ is non-degenerate. 
Then $\UP$ is a compact subset of $L_2(\TTd)$ if and only if 
\eqref{e:hanoi2014-05-22} is satisfied and $0\in A$.
\end{theorem}
\begin{proof}
Let us prove that there exists $\gamma_1\in\RPP$ such that
\begin{equation} 
\label{eq[|P(x)|<]}
\big(\forall k\in\ZZd\big)\quad |P(k)|\leq\gamma_1
\max_{\alpha\in\vartheta(A)}|k^\alpha|. 
\end{equation}
Since there exists $\gamma_1\in\RPP$ such that
\begin{equation} 
\big(\forall k\in\ZZd\big)\quad |P(k)|\leq\gamma_1
\max_{\alpha\in A}|k^\alpha|, 
\end{equation}
and since \eqref{eq[|P(x)|<]} trivially holds if there exists 
$j\in \{1,\ldots,d\}$ such that $k_j=0$, it is enough to show that
\begin{equation} 
\big(\forall \alpha \in A\big) \big(\forall k\in\NN^{*d}\big)\quad 
k^\alpha\leq\max_{\beta\in\vartheta(A)} k^\beta,
\end{equation}
and a fortiori that
\begin{equation} 
\big(\forall \alpha\in A\big)\big(\forall x\in\RRd_+\big)\quad 
\scal{\alpha}{x}\leq\max_{\beta\in\vartheta(A)}\scal{\beta}{x}.
\end{equation}
Indeed, since $\alpha \in \conv{\vartheta(A)}$, by
Carath\'eodory's theorem \cite[Theorem~17.1]{Rock70}, 
$\alpha$ is a convex combination of points 
$(\beta^j)_{1\leq j\leq d+1}$ in $\vartheta(B)$, say
\begin{equation}
\alpha =
\sum_{j=1}^{d+1}\lambda_j\beta^j,\quad
\text{where}\quad (\lambda_j)_{1\leq j\leq d+1}\in\RP^{d+1}
\quad\text{and}\quad\sum_{j=1}^{d+1} \lambda_j=1.
\end{equation}
Therefore
\begin{equation}
\big(\forall x\in\RRd_+\big)\quad\scal{\alpha}{x}=
\sum_{j=1}^{d+1} \lambda_j \scal{\beta_j}{x}
\leq \sum_{j=1}^{d+1}\lambda_j\max_{\beta\in\vartheta(A)} 
\scal{\beta}{x}
=\max_{\beta\in\vartheta(A)}\scal{\beta}{x}.
\end{equation}
Hence, Lemma~\ref{l:non-deg} asserts that there exists 
$\gamma_2\in\RPP$ such that
\begin{equation}
\label{e:2013-09-18a}
\big(\forall k\in\ZZd\big)\quad 
\gamma_2 \max_{\alpha\in\vartheta(A)}|k^\alpha|\leq
|P(k)|\leq \gamma_1\max_{\alpha\in\vartheta(A)}|k^\alpha|.
\end{equation}
Consequently, by Lemma~\ref{l:7}, 
$\UP$ is a compact set in $L_2(\TTd)$ if and only if, for every
$t\in\RP$, $\Omega_A(t)$ is finite and 
\begin{equation}
\inf_{k\in\ZZdp}\max_{\alpha\in A} k^\alpha>0.
\end{equation}
In view of Lemma~\ref{l:6}, the first condition is equivalent 
to \eqref{e:hanoi2014-05-22} and the second to $0\in A$.
\end{proof}

\section{Main result}
\label{Widths}

\subsection{Convex-analytical results}

Several important convex-analytical facts underly our analysis 
(see \cite{Livre1,Rock70} for background on convex analysis).
We start with the following corollary. 

\begin{corollary} 
\label{c:12} 
Suppose that $P(D)$ is non-degenerate. Then 
$\big(\forall k\in\ZZd\big)$
$|P(k)|\asymp\max_{\alpha\in\vartheta(A)}|k^\alpha|$.
\end{corollary}
\begin{proof}
Combine \eqref{e:2013-09-18a} and Lemma~\ref{l:non-deg}.
\end{proof}
 
Next, we investigate the geometry of our problem from the 
view-point of convex duality. Let $C$ be a subset of $\RR^d$.
Recall that the \emph{polar set}\index{polar set} of $C$ is
\begin{equation}
\label{e:polarset}
C^\odot=\menge{x\in\RR^d}{(\forall\alpha\in C)\;\;
\scal{\alpha}{x}\leq 1},
\end{equation}
and the \emph{indicator function} of $C$ is 
\begin{equation}
\label{e:iota}
\iota_C\colon\RR^d\to\RX\colon x\mapsto 
\begin{cases}
0,&\text{if}\;\:x\in C;\\
\pinf,&\text{otherwise.}
\end{cases}
\end{equation}
Moreover, if $C$ is convex and $0\in C$, the 
\emph{Minkowski gauge} of $C$ is the lower semicontinuous convex
function 
\begin{equation}
m_C\colon\RR^d\to\RX\colon x\mapsto\text{\rm inf}\:
\menge{\xi\in\RPP}{x\in\xi C}.
\end{equation}
Finally, the domain of a function $\varphi\colon\RR^d\to\RX$ is
$\dom\varphi=\menge{x\in\RR^d}{\varphi(x)<\pinf}$.

\begin{lemma}
\label{l:hanoi2014-05-27}
Let $B$ be a nonempty finite subset of $\RRdp$ such that 
\begin{equation}
\label{e:hanoi2014-05-24}
0\in B\quad\text{and}\quad(\forall j\in\{1,\ldots,d\})\quad
B\cap\mathcal{R}^j\neq\varnothing.
\end{equation}
Set $\boldsymbol{1}=(1,\ldots,1)\in\RR^d$, 
let $\mu(B)$ be the optimal value of the problem 
\begin{equation} 
\label{e:hanoi-dual}
\maximize{x\in B^\odot}{\sum_{j=1}^d x_j}, 
\end{equation}
and set 
\begin{equation} 
\label{e:hanoi-primal}
\varrho(B)=\text{\rm max}\menge{\rho\in\RPP}
{\rho\boldsymbol{1}\in\conv{B}}.
\end{equation}
Then $\varrho(B)\in\RPP$ and $\mu(B)=1/\varrho(B)$.
\end{lemma}
\begin{proof}
It follows from \eqref{e:hanoi2014-05-24} that 
\begin{equation}
\label{e:2014-05-31}
\RR_+^d\cap B^\odot=\RR_+^d\cap\bigcap_{\alpha\in B}
\menge{x\in\RR^d}{\scal{x}{\alpha}\leq 1}
\end{equation}
is a nonempty compact set and hence \eqref{e:hanoi-dual} does 
have a solution. Now fix $j\in\{1,\ldots,d\}$. Then
$(\exi a_j\in\RPP)$ $a_ju^j\in B$. Hence $x^j=(1/a_j)u^j\in B^\odot$ 
and therefore $\mu(B)=\max_{x\in B^\odot}\scal{x}{\boldsymbol{1}}
\geq\scal{x^j}{\boldsymbol{1}}=1/a_j>0$. 
Altogether $\mu(B)\in\RPP$. Likewise,
\eqref{e:hanoi2014-05-24} implies that $\varrho(B)\in\RPP$. 
Let us set $\varphi=m_{\conv{B}}$ and 
$\psi=\iota_{\{\boldsymbol{1}\}}$. 
Then it follows from \eqref{e:hanoi2014-05-24} that
$\dom \varphi=\dom m_{\conv{B}}=\RP^d$.
Furthermore, the conjugate of $\varphi$ is 
$\varphi^*=\iota_{(\conv{B})^\odot}=\iota_{B^\odot}$ 
\cite[Propositions~14.12 and 7.14(vi)]{Livre1} and 
the conjugate of $\psi$ is 
$\psi^*=\scal{\cdot}{\boldsymbol{1}}$. Hence, since 
$\boldsymbol{1}\in\inte\,\dom\varphi=\RPP^d$, 
$\dom\psi\cap\inte\,\dom\varphi\neq\varnothing$ and the Fenchel 
duality formula \cite[Proposition~15.13]{Livre1} yields
\begin{align}
\label{e:hanoi2014-05-27u}
\mu(B)
&=\max_{x\in B^\odot}\sum_{j=1}^d{x_j}\nonumber\\
&=-\min_{x\in B^\odot}\scal{-x}{\boldsymbol{1}}\nonumber\\
&=-\min_{x\in\RR^d}\big(\iota_{B^\odot}(x)+\scal{-x}
{\boldsymbol{1}}\big)
\nonumber\\
&=-\min_{x\in\RR^d}\big(\varphi^*(x)+\psi^*(-x)\big)\nonumber\\
&=\inf_{\alpha\in\RR^d}\big(\varphi(\alpha)+
\psi(\alpha)\big)\nonumber\\
&=\inf_{\alpha\in\RR^d}\big(m_{\conv{B}}(\alpha)+
\iota_{\{\boldsymbol{1}\}}(\alpha)\big)\nonumber\\
&=m_{\conv{B}}(\boldsymbol{1})\nonumber\\
&=\inf\:\menge{\xi\in\RPP}{\boldsymbol{1}\in\xi\conv{B}}
\nonumber\\
&=\Frac{1}{\sup\:
\menge{\rho\in\RPP}{\rho\boldsymbol{1}\in\conv{B}}}.
\end{align}
We conclude that $\mu(B)=1/\varrho(B)$.
\end{proof}

To illustrate the duality principles underlying
Lemma~\ref{l:hanoi2014-05-27}, we consider two examples.

\begin{example}
\label{ex:hanoi2014-05-24a}
We consider the case when $d=2$ and  
$B=\{(6,0),(0,6),(4,4),(0,0)\}$ (see Figure~\ref{fig:2014-05-24a}). 
Then \eqref{e:hanoi2014-05-24} is satisfied, $\mu(B)=1/4$, and 
$\varrho(B)=4$. 
The set of solutions to \eqref{e:hanoi-dual}
is the set $S$ represented by the solid red segment:
$S=\menge{(x_1,x_2)\in[1/12,1/6]^2}{x_1+x_2=1/4}$. 

\begin{figure}
\centering
\scalebox{0.80} % Change this value to rescale the drawing.
{
\begin{pspicture}(-2,-2)(9,9) % extreme points of the figure
\pspolygon[fillstyle=solid,fillcolor=lightgray!30]%
(0.0,0.0)(6.0,0.0)(4.0,4.0)(0.0,6.0)
\psline[linewidth=0.05cm,arrowsize=0.3cm,linestyle=solid]{->}%
(-1.4,0)(8,0)
\psline[linewidth=0.05cm,arrowsize=0.3cm,linestyle=solid]{->}%
(0,-1.4)(0,8)
\psline[linewidth=0.08cm,linestyle=solid]{-}(0.0,0.0)(0.0,6.0)
\psline[linewidth=0.08cm,linestyle=solid]{-}(0.0,0.0)(6.0,0.0)
\psline[linewidth=0.08cm,linestyle=solid]{-}(0.0,6.0)(4.0,4.0)
\psline[linewidth=0.08cm,linestyle=solid]{-}(4.0,4.0)(6.0,0.0)
\psline[linecolor=blue,linewidth=0.08cm,linestyle=dashed,%
arrowsize=0.5]{->}(0.0,0.0)(3.94,3.94)
\psdot[linecolor=black,linewidth=0.09cm](0.0,0.0)
\psdot[linecolor=black,linewidth=0.09cm](0.0,6.0)
\psdot[linecolor=black,linewidth=0.09cm](6.0,0.0)
\psdot[linecolor=blue,linewidth=0.09cm](4.0,4.0)
\rput(-0.5,6.0){\Large${6}$}
\rput(-0.5,4.0){\Large${4}$}\rput(0.0,4.0){\large${-}$}
\rput(4.0,-0.5){\Large${4}$}\rput(4.0,0.0){\large${|}$}
\rput(6.0,-0.5){\Large${6}$}
\rput(-0.4,-0.4){\Large${0}$}
\rput(4.8,4.5){\Large$\varrho(B)\boldsymbol{1}$}
\rput(3.4,1.1){\Large$\conv{B}$}
\end{pspicture} 
}

\scalebox{2.0} % multiply by 10 all intended values below
{
\begin{pspicture}(-1.4,-1.4)(4,4) % extreme points of the figure
\pspolygon[linewidth=0.0mm,fillstyle=solid,fillcolor=lightgray!30]%
(-0.5,1.667)(0.8333,1.667)(1.667,0.8333)(1.667,-0.55)(-0.5,-0.55)
\psline[linewidth=0.2mm,arrowsize=1.2mm,linestyle=solid]{->}%
(-0.5,0)(3.3,0)
\psline[linewidth=0.2mm,arrowsize=1.2mm,linestyle=solid]{->}%
(0,-0.55)(0,3.3)
\psline[linewidth=0.2mm,linestyle=dashed]{-}%
(-0.5,1.667)(3.0,1.667)
\psline[linewidth=0.2mm,linestyle=dashed]{-}%
(1.667,-0.5)(1.667,3.0)
\psline[linewidth=0.2mm,linestyle=dashed]{-}(-0.5,3.0)(3.0,-0.5)
\psline[linewidth=0.3mm,linestyle=solid]{-}%
(-0.5,1.667)(0.8333,1.667)
\psline[linewidth=0.3mm,linestyle=solid]{-}%
(1.667,-0.55)(1.667,0.8333)
\psline[linewidth=0.3mm,linestyle=solid]{-}%
(0.8333,1.667)(1.667,0.8333)
\psline[linecolor=red,linewidth=0.4mm,linestyle=dotted,dotsep=2pt]%
{-}(-0.5,3.0)(3.0,-0.5)
\psline[linecolor=red,linewidth=0.3mm,linestyle=solid]{-}%
(0.8333,1.667)(1.667,0.8333)
\rput(-0.30,0.8333){\tiny$\frac{1}{12}$}\rput(0,0.8333){\tiny$-$}
\rput(-0.40,2.5){\tiny${\mu(B)}$}\rput(0,2.5){\tiny$-$}
\rput(-0.30,1.87){\tiny$\frac{1}{6}$}
\rput(2.7,0.30){\tiny$\mu(B)$}\rput(2.5,0){\tiny$|$}
\rput(1.80,-0.35){\tiny$\frac{1}{6}$}\rput(0.8333,0){\tiny$|$}
\rput(0.8333,-0.35){\tiny$\frac{1}{12}$}\rput(0.8333,0){\tiny$|$}
\rput(0.8333,0.8333){\tiny$B^{\odot}$}
\end{pspicture} 
}
\caption{Graphical illustration of 
Example~\ref{ex:hanoi2014-05-24a}: In gray, the Newton polyhedron 
(top) and its polar (bottom). The dashed lines are the hyperplanes
delimiting the polar set $B^\odot$ and the dotted line represents 
the optimal level curve of the objective function 
$x\mapsto\scal{x}{\boldsymbol{1}}$ in \eqref{e:hanoi-dual}. The
solid red segment depicts the solution set of \eqref{e:hanoi-dual}.}
\label{fig:2014-05-24a}
\end{figure}
\end{example}

\begin{example}
\label{ex:hanoi2014-05-24b}
In this example we consider the case when 
$B=\{(0,6), (2,4), (4,0), (0,0)\}$.
Then \eqref{e:hanoi2014-05-24} is satisfied, $\mu(B)=3/8$, and
$\varrho(B)=8/3$. 
The set of solutions to \eqref{e:hanoi-dual} reduces to the 
singleton $S=\{(1/4,1/8)\}$.

\begin{figure}
\centering
\scalebox{0.80} % Change this value to rescale the drawing.
{
\begin{pspicture}(-2,-2)(7,7) % extreme points of the figure
\pspolygon[fillstyle=solid,fillcolor=lightgray!30]%
(0.0,0.0)(0.0,6.0)(2.0,4.0)(4.0,0.0)
\psline[linewidth=0.05cm,arrowsize=0.3cm,linestyle=solid]{->}%
(-1.4,0)(6,0)
\psline[linewidth=0.05cm,arrowsize=0.3cm,linestyle=solid]{->}%
(0,-1.4)(0,8)
\psline[linewidth=0.08cm,linestyle=solid]{-}(0.0,0.0)(0.0,6.0)
\psline[linewidth=0.08cm,linestyle=solid]{-}(0.0,6.0)(2.0,4.0)
\psline[linecolor=blue,linewidth=0.08cm,linestyle=solid]{-}%
(2.0,4.0)(4.0,0.0)
\psline[linewidth=0.08cm,linestyle=solid]{-}(4.0,0.0)(0.0,0.0)
\psdot[linecolor=black,linewidth=0.09cm](0.0,0.0)
\psdot[linecolor=black,linewidth=0.09cm](0.0,6.0)
\psdot[linecolor=black,linewidth=0.09cm](2.0,4.0)
\psdot[linecolor=black,linewidth=0.09cm](4.0,0.0)
\psline[linecolor=blue,linewidth=0.08cm,linestyle=dashed,%
arrowsize=0.5]{->}(0.0,0.0)(2.667,2.667)
\rput(-0.5,6.0){\Large${6}$}
\rput(-0.5,4.0){\Large${4}$}\rput(0.0,4.0){\large${-}$}
\rput(2.0,-0.5){\Large${2}$}\rput(2.0,0.0){\large${|}$}
\rput(4.0,-0.5){\Large${4}$}
\rput(-0.4,-0.4){\Large${0}$}
\rput(1.2,3.0){\Large$\conv{B}$}
\rput(3.4,2.94){\Large$\varrho(B)\boldsymbol{1}$}
\end{pspicture} 
}

\scalebox{1.7} % multiply by 10 all intended values below
{
\begin{pspicture}(-1,-1)(7,5) % extreme points of the figure
\pspolygon[linewidth=0.0mm,fillstyle=solid,fillcolor=lightgray!30]%
(-0.5,1.667)(1.667,1.667)(2.5,1.25)(2.5,-0.65)(-0.5,-0.65)
\psline[linewidth=0.2mm,arrowsize=1.2mm,linestyle=solid]{->}%
(-0.5,0)(6.5,0)
\psline[linewidth=0.2mm,arrowsize=1.2mm,linestyle=solid]{->}%
(0,-0.65)(0,3.3)
\psline[linewidth=0.2mm,linestyle=dashed]{-}%
(-0.5,1.667)(6.1,1.667)
\psline[linewidth=0.2mm,linestyle=dashed]{-}%
(2.5,-0.55)(2.5,3.0)
\psline[linewidth=0.3mm,linestyle=solid]{-}%
(-0.5,1.667)(1.667,1.667)
\psline[linecolor=red,linewidth=0.4mm,linestyle=dotted,dotsep=2pt]%
(0.625,3.125)(4.375,-0.625)
\psline[linewidth=0.3mm,linestyle=solid]{-}%
(2.5,-0.65)(2.5,1.25)
\psline[linewidth=0.3mm,linestyle=solid]{-}%
(1.667,1.667)(2.5,1.25)
\psline[linewidth=0.2mm,linestyle=dashed]{-}(-0.5,2.75)(6.1,-0.55)
\rput(-0.30,1.25){\small$\frac{1}{8}$}\rput(0,1.25)%
{\small$-$}
\rput(0.20,2.7){\small${\frac{1}{4}}$}\rput(0,2.5)%
{\small$-$}
\rput(-0.30,1.97){\small$\frac{1}{6}$}
\rput(2.65,-0.30){\small$\frac{1}{4}$}\rput(2.5,0)%
{\tiny$|$}
\rput(3.65,-0.40){\scriptsize$\mu(B)$}\rput(3.75,0)%
{\tiny$|$}
\rput(5.0,-0.40){\small$\frac{1}{2}$}\rput(5.0,0)%
{\tiny$|$}
\rput(1.667,-0.37){\small$\frac{1}{6}$}\rput(1.667,0){\tiny$|$}
\rput(1.0,0.8333){\scriptsize$B^{\odot}$}
\psdot[linecolor=red,linewidth=0.03cm](2.5,1.25)
\end{pspicture} 
}

\caption{Graphical illustration of 
Example~\ref{ex:hanoi2014-05-24b}: In gray, the Newton polyhedron 
(top) and its polar (bottom). The dashed lines are the hyperplanes
delimiting the polar set $B^\odot$ and the dotted line represents 
the optimal level curve of the objective function 
$x\mapsto\scal{x}{\boldsymbol{1}}$ in \eqref{e:hanoi-dual}. The red
dot locates the unique solution to \eqref{e:hanoi-dual}.}
\label{fig:2014-05-24b}
\end{figure}
\end{example}

\begin{lemma} 
\label{lemma[asymp]} 
Let $B$ be a nonempty finite subset of $\RRdp$ and suppose that
\begin{equation}
\label{e:uj}
(\forall j\in\{1,\ldots,d\})\quad
B\cap\mathcal{R}^j\neq\varnothing.
\end{equation}
Let $\mu(B)$ be the optimal value of the problem 
\begin{equation} 
\label{e:hanoi-dual2}
\maximize{x\in B^\odot}{\sum_{j=1}^d x_j}, 
\end{equation}
and let $\nu(B)$ be the dimension of its set of solutions.
Then $\mu(B)\in\RPP$ and 
\begin{equation}
(\forall t\in [2,\pinf[)\quad
\card\Omega_B(t)\asymp t^{\mu(B)}\big(\log t\big)^{\nu(B)}.
\end{equation}
\end{lemma}
\begin{proof} 
The fact that $\mu(B)\in\RPP$ was proved as in 
Lemma~\ref{l:hanoi2014-05-27}. Now fix $t\in [2,\pinf[$ and set
$\Lambda_B(t)=\menge{x\in\RRdp}{\text{\rm max}_{\alpha\in B}\,
x^\alpha\leq t}$.
Then, as in the proof of Lemma~\ref{l:6}, one
can see that $\Lambda_B(t)$ is a bounded subset of $\RRdp$. 
If we denote by $\text{vol}\,\Lambda_B(t)$
the volume of $\Lambda_B(t)$, then it follows from 
\cite[Theorem~1]{DD84} that
\begin{equation}
\text{vol}\,\Lambda_B(t)\asymp t^{\mu(B)}(\log t)^{\nu(B)}.
\end{equation}
Furthermore, proceeding as in the proof of \cite[Theorem~2]{DD84}, 
one shows that
\begin{equation}
\card\Omega_B(t)\asymp\text{vol}\,\Lambda_B(t).
\end{equation}
These asymptotic relations prove the claim.
\end{proof}

\subsection{Main result: asymptotic order of Kolmogorov $n$-width}

Our main result can now be stated and proved.

\begin{theorem} 
\label{t:1} 
Suppose that $P(D)$ is non-degenerate and that 
\begin{equation}
\label{e:hanoi2014-05-26}
0\in A\quad\text{and}\quad(\forall j\in\{1,\ldots,d\})\quad
A\cap\mathcal{R}^j\neq\varnothing.
\end{equation}
Let $\mu$ be the optimal value of the problem 
\begin{equation} 
\label{e:hanoi-dual3}
\maximize{x\in \vartheta(A)^\odot}{\sum_{j=1}^d x_j}, 
\end{equation}
let $\nu$ be the dimension of its set of solutions, and set
\begin{equation} 
\label{e:hanoi-primal2}
\varrho=\text{\rm max}\menge{\rho\in\RPP}
{\rho\boldsymbol{1}\in\conv{\vartheta(A)}}.
\end{equation}
Then $\mu=1/\varrho\in\RPP$ and, for $n$ sufficiently large, 
\begin{equation} 
\label{asymp[d_n]}
d_n\big(\UP,L_2(\TTd)\big)\asymp 
n^{-\varrho}\big(\log n\big)^{\nu\varrho}.
\end{equation}
Equivalently, using \eqref{e:2}, for $\varepsilon\in\RPP$ 
sufficiently small,
\begin{equation} 
\label{asymp[n_e]}
n_\varepsilon \big(\UP,L_2(\TTd)\big)\asymp 
\varepsilon^{-1/\varrho}|\log\varepsilon|^{\nu}.
\end{equation}
\end{theorem}
\begin{proof}
Since $A$ satisfies \eqref{e:hanoi2014-05-26}, so does 
$\vartheta(A)$. Hence the fact that $\mu=1/\varrho\in\RPP$ 
follows from Lemma~\ref{l:hanoi2014-05-27}. We also note that the 
equivalence between \eqref{asymp[d_n]} and \eqref{asymp[n_e]}
follows from \eqref{e:1} and \eqref{e:2}. 
To show \eqref{asymp[d_n]}, set $\bar{t}=\max\{2,\tau\}$.
Then we derive from Corollary~\ref{c:12} that
\begin{equation}
(\forall t\in[\bar{t},+\infty[)\quad 
\card\Omega_{\vartheta(A)}(t)\asymp\card K(t).
\end{equation}
Applying Lemma~\ref{lemma[asymp]} to $\vartheta(A)$ yields
\begin{equation}
\label{e:hanoi}
(\forall t\in[\bar{t},\pinf[)\quad
\dim V(t)=\card K(t)\asymp 
t^{1/\varrho}\big(\log t\big)^\nu.
\end{equation}
Hence, for every $n\in\NN$ large enough, there exists $t\in\RPP$ 
depending on $n$ such that
\begin{multline} 
\label{ineq[<n<]}
\gamma_1\dim V(t)\leq\gamma_3t^{1/\varrho}\big(\log t\big)^\nu
\leq n<\gamma_3(t+1)^{1/\varrho}\big(\log (t+1)\big)^\nu\\
\leq\gamma_2\dim V(t+1)\leq\gamma_4
t^{1/\varrho}\big(\log t\big)^\nu,
\end{multline}
where $\gamma_1$, $\gamma_2$, $\gamma_3$, and $\gamma_4$ are 
strictly positive real parameters that are independent from $n$ 
and $t$. Therefore,
\begin{equation} 
\label{e:24}
n\asymp t^{1/\varrho}\big(\log t\big)^\nu. 
\end{equation}
or, equivalently,
\begin{equation} 
\label{ineq[t^{-1}]}
t^{-1} \asymp n^{-\varrho}\big(\log n\big)^{\nu\varrho}.
\end{equation}
It therefore follows from \eqref{e:1} and 
Corollary~\ref{corollary[|f-S_t(f)|<]} that
\begin{equation}
d_n\big(\UP,L_2(\TTd)\big)\leq
t^{-1} \asymp n^{-\varrho}\big(\log n\big)^{\nu\varrho},
\end{equation}
which establishes the upper bound in \eqref{asymp[d_n]}.
To establish the lower bound, let us recall from \cite{Ti60}
that, for every ${n+1}$-dimensional vector subspace $G_{n+1}$ of
$L_2(\TTd)$ and every $\eta \in\RPP$, we have 
\begin{equation}
\label{e:Ti60}
d_n\big(B_{n+1}(\eta),L_2(\TTd)\big)=\eta,\quad\text{where}\quad
B_{n+1}(\eta)=\menge{f\in G_{n+1}}{\|f\|_{L_2(\TTd)}\leq\eta}.
\end{equation}
Arguing as in \eqref{e:hanoi}--\eqref{ineq[t^{-1}]}, 
for $n\in\NN$ sufficiently large, there exists $t\in\RPP$ 
such that
\begin{equation} 
\label{ineq[dimV(t)>]}
\dim V(t)\geq\gamma_5 t^{1/\varrho}\big(\log t\big)^\nu > 
n\geq\gamma_6 t^{1/\varrho}\big(\log t\big)^\nu,
\end{equation}
where $\gamma_5\in\RPP$ and $\gamma_6\in\RPP$ are independent from 
$n$ and $t$. Now set
\begin{equation}
\label{e:stevens}
U(t)=\menge{f\in V(t)}{\|f\|_2\leq t^{-1}}.
\end{equation}
By Lemma~\ref{l:B}, $U(t)\subset\UP$. Consequently, it follows
from \eqref{e:Ti60}--\eqref{e:stevens} and \eqref{ineq[t^{-1}]} that
\begin{equation} 
d_n\big(\UP,L_2(\TTd)\big)\geq d_n\big(U(t),L_2(\TTd)\big) 
\geq t^{-1}\asymp n^{-\varrho}\big(\log n\big)^{\nu\varrho},
\end{equation}
which concludes the proof of \eqref{asymp[d_n]}. Next, let
us prove \eqref{asymp[n_e]}. Given a 
sufficiently small $\varepsilon\in\RPP$, take 
$t\in\RPP$ such that
$0 < t-1 <\varepsilon^{-1}\leq t$ and $\dim V(t)>1$. From
the above results, it can be seen that 
\begin{equation} 
\label{ineq[n_e]}
\dim V(t) - 1\leq n_\varepsilon\big(\UP,L_2(\TTd)\big)
\leq \dim V(t)
\end{equation}
which, together with \eqref{e:hanoi}, proves \eqref{asymp[n_e]}.
\end{proof}

\begin{remark}
We have actually proven a bit more than Theorem~\ref{t:1}. 
Namely, suppose that $P(D)$ satisfies the conditions of compactness
for $\UP$ stated in Lemma~\ref{l:7} and, for 
every $n\in\NN$, let $t(n)$ be the largest number such 
that $\card K(t(n))\leq n$.
Then, for $n$ sufficiently large, we have
\begin{equation} 
\label{e:7g7}
d_n\big(\UP,L_2(\TTd)\big) \asymp \frac{1}{t(n)}.
\end{equation}
\end{remark}

\section{Examples} 
\label{Norms equivalences}

We first establish norm equivalences and use them to provide 
examples of asymptotic orders of $d_n\big(\UP, L_2(\TTd)\big)$ 
for non-degenerate and degenerate differential operators.

\begin{theorem} 
\label{t:2} 
Suppose that $P(D)$ is non-degenerate and set
\begin{equation} 
Q\colon x\mapsto\sum_{\alpha\in\vartheta(A)} x^\alpha. 
\end{equation}
Then 
\begin{equation} 
\label{asymp[norm-equivalence]}
\big(\forall f\in W_2^{[P]}\big)\quad
\|f\|_{\WP}^2 \asymp \|f\|_{W^{[Q]}_2}^2
\asymp \sum_{\alpha\in\vartheta(A)}\|D^{\alpha}f\|_2^2
\asymp \max_{\alpha\in\vartheta(A)}\|D^{\alpha}f\|_2^2.
\end{equation}
Moreover, the seminorms in \eqref{asymp[norm-equivalence]} are 
norms if and only if $0\in A$.
\end{theorem}
\begin{proof}
Let $f\in W_2^{[P]}$. It is clear that
\begin{equation}
\sum_{\alpha\in\vartheta(A)}\|D^{\alpha}f\|_2^2
\asymp\max_{\alpha\in\vartheta(A)}\|D^{\alpha}f\|_2^2.
\end{equation}
Parseval's identity and Corollary~\ref{c:12} yield
\begin{equation} 
\max_{\alpha\in\vartheta(A)}\|D^{\alpha}f\|_2^2
=\max_{\alpha\in\vartheta(A)}\sum_{k\in\ZZd}|k|^{2\alpha}
|\hat{f}(k)|^2\leq
\sum_{k\in\ZZd}\Big(\max_{\alpha\in\vartheta(A)}|k^\alpha|\Big)^2
|\hat{f}(k)|^2.
\end{equation}
Now let $(\ZZd(\alpha))_{\alpha\in\vartheta(A)}$ be a partition of 
$\ZZd$ such that
\begin{equation}
\max_{\beta\in\vartheta(A)}|k^\beta|=|k^\alpha|, 
\quad k\in\ZZd(\alpha).
\end{equation}
Then 
\begin{equation} 
\begin{aligned}
\max_{\alpha\in\vartheta(A)}\|D^{\alpha}f\|_2^2 
&=\max_{\alpha\in\vartheta(A)}\sum_{\alpha'\in\vartheta(A)}
\sum_{k\in\ZZd(\alpha')}|k^{2 \alpha}|\,|\hat{f}(k)|^2 \\[1.5ex]
&\geq\sum_{\alpha'\in\vartheta(A)}\sum_{k\in\ZZd(\alpha')}
|k^{2\alpha'}|\,|\hat{f}(k)|^2 \\[1.5ex]
&=\sum_{k\in\ZZd}\max_{\alpha\in\vartheta(A)}|k^\alpha|^2
\,|\hat{f}(k)|^2.
\end{aligned}
\end{equation}
Thus, 
\begin{equation} 
\max_{\alpha\in\vartheta(A)}\|D^{\alpha}f\|_2^2
=\sum_{k\in\ZZd}\max_{\alpha\in\vartheta(A)}|k^\alpha|^2 
\,|\hat{f}(k)|^2.
\end{equation}
Hence, appealing to Corollary~\ref{c:12} and \eqref{e:parseval}, 
we obtain 
\begin{equation} 
\max_{\alpha\in\vartheta(A)}\|D^{\alpha}f\|_2^2
\asymp 
\|f\|_{\WP}^2.
\end{equation}
The relation
\begin{equation} 
\max_{\alpha\in\vartheta(A)}\|D^{\alpha}f\|_2^2\asymp 
\|f\|_{W^{[Q]}_2}^2
\end{equation}
follows from the last seminorm equivalence and the identity
$\vartheta(\vartheta(A))=\vartheta(A)$. 
Therefore, we derive from \eqref{asymp[norm-equivalence]} that the
seminorms in \eqref{asymp[norm-equivalence]} are norms if and
only if $0\in A$.
\end{proof}

\subsection{Isotropic Sobolev classes}

Let $s\in\NN^*$.
The isotropic Sobolev space $H^s$ is the Hilbert
space of functions $f \in L_2(\TTd)$ equipped with the norm 
\begin{equation}
\|\cdot\|_{H^s}\colon f\mapsto
\sqrt{\|f\|_2^2+\sum_{|\alpha|=s} \|f^{(\alpha)}\|_2^2}.
\end{equation}
Consider
\begin{equation}
P\colon x\mapsto 1+\sum_{|\alpha|=s} x^\alpha =
\sum_{\alpha\in A} x^\alpha,
\end{equation}
where $A=\{0\}\cup\menge{\alpha\in\NNd}{|\alpha|=s}$. 
If $s$ is even, it follows directly from Lemma~\ref{l:non-deg} that 
the differential operator $P(D)$ is non-degenerate,
and consequently, by Theorem~\ref{t:2}, $\|\cdot\|_{H^s}$ is
equivalent to one of the norms appearing in 
\eqref{asymp[norm-equivalence]} with
$\vartheta(A)=\{0\}\cup\menge{su^j}{1\leq j\leq d}$ and
\begin{equation}
Q\colon x\mapsto 1+\sum_{j=1}^d x_j^s.
\end{equation}
Moreover, we have $\varrho(A)= s/d$ and $\nu(a)=0$. Therefore, 
we retrieve from Theorem~\ref{t:1} the well-known result
\begin{equation} 
d_n\big(U^s,L_2(\TTd)\big)\asymp n^{-s/d},
\end{equation}
where $U^s$ denotes the closed unit ball in $H^s$. 
This result is a direct 
generalization of the first result on $n$-widths established by 
Kolmogorov in \cite{Ko36}.

\subsection{Anisotropic Sobolev classes}

Given $\beta=(\beta_1,\ldots,\beta_d)\in\NN^{*d}$, the anisotropic 
Sobolev space $H^\beta$ is the Hilbert space of functions 
$f\in L_2$ equipped with the norm 
\begin{equation}
\|\cdot\|_{H^\beta}^2\colon f\mapsto
\sqrt{\|f\|_2^2+\sum_{j=1}^d\|f^{(\beta_ju^j)}\|_2^2}.
\end{equation}
Consider the polynomial
\begin{equation}
P\colon x\mapsto
1+\sum_{j=1}^d x_j^{\beta_j}=\sum_{\alpha \in A} x^\alpha,
\end{equation}
where $A=\{0\}\cup\menge{\beta_ju^j}{1\leq j\leq d}$. If the
coordinates of $\beta$ are even, the differential operator $P(D)$
is non-degenerate. Consequently, by Theorem~\ref{t:2}, 
$\|\cdot\|_{H^\beta}$ is equivalent to one of the norms in
\eqref{asymp[norm-equivalence]} with $\vartheta(A)= A$ and
\begin{equation}
Q=P.
\end{equation}
We have 
\begin{equation}
\varrho=\varrho(A)=\left(\sum_{j=1}^d 1/\beta_j \right)^{-1}
\end{equation} 
and $\nu(A)=0$, and therefore, from 
Theorem~\ref{t:1} we retrieve the known result \cite{Ho80}
\begin{equation} 
d_n\big(U^\beta,L_2(\TTd)\big) \asymp n^{-\varrho},
\end{equation}
where  $U^\varrho$ denotes the unit ball in in $H^\beta$.

\subsection{Classes of functions with a bounded mixed derivative} 

Let $\alpha=(\alpha_1,\ldots,\alpha_d)\in\NNd$ with 
$0<\alpha_1=\cdots=\alpha_{\nu+1}<\alpha_{\nu+2}= \cdots 
=\alpha_d$ for some $\nu\in\{0,\ldots,d-1\}$.
Given a set $e\subset \{1,\ldots,d\}$, let the vector
$\alpha(e)\in\ZZdp$ be defined by
$\alpha(e)_j=\alpha_j$ if $j\in e$, and $\alpha(e)_j=0$
otherwise (in particular, $\alpha(\varnothing)=0$ and
$\alpha(\{1,\ldots,d\})=\alpha$).  
The space $W^\alpha_2$ is the Hilbert space of functions
$f\in L_2$ equipped with the norm 
\begin{equation}
\|\cdot\|_{W^\alpha_2}\colon f\mapsto 
\sqrt{\sum_{e \subset\{1,\ldots,d\}}\|f^{(\alpha(e))}\|_2^2}.
\end{equation}
Consider
\begin{equation}
P\colon x\mapsto
\sum_{e \subset \{1,\ldots,d\}} x^{\alpha(e)}
=\sum_{\alpha \in A} x^\alpha,
\end{equation}
where $A= \menge{\alpha(e)}{e\subset \{1,\ldots,d\}}$.
If the coordinates of $\alpha$ are even, the differential operator
$P(D)$ is non-degenerate and hence, by Theorem~\ref{t:2}, 
$\|\cdot\|_{W^\alpha_2}$ is
equivalent to one of the norms in \eqref{asymp[norm-equivalence]}
with $\vartheta(A)= A$ and $Q=P$. We have $\varrho(A)=\alpha_1$ 
and $\nu(A)=\nu$, and therefore,  from 
Theorem~\ref{t:1} we recover the result proven in
\cite{Ba60a}, namely that for $n$ sufficiently large
\begin{equation} 
d_n\big(U^\alpha_2,L_2(\TTd)\big) \asymp
n^{-\alpha_1}\big(\log n\big)^{\nu \alpha_1},
\end{equation}
where $U^\alpha_2$ denotes the unit ball in $W^\alpha_2$.
In the particular case when $\alpha=\varrho{\bf 1}$, we have 
\begin{equation} 
d_n\big(U^{\varrho{\bf 1}}_2,L_2(\TTd)\big) \asymp n^{-\varrho} 
\big(\log n\big)^{(d-1)\varrho}.
\end{equation}

\subsection{Classes of functions with several bounded 
mixed derivatives} 
\label{W^A_2}
Suppose that \eqref{e:hanoi2014-05-26} is satisfied.
Let $W^A_2$ be the Hilbert space of functions 
$f\in L_2(\TTd)$ equipped with the norm 
\begin{equation}
\|\cdot\|_{W^A_2}\colon f\mapsto 
\sqrt{\sum_{\alpha \in A}\|f^{(\alpha)}\|_2^2}.
\end{equation}
Notice that spaces $H^s$, $H^r$, and $W^\alpha_2$ are a particular
cases of $W^A_2$. Now consider
\begin{equation}
P\colon x\mapsto\sum_{\alpha \in A} x^\alpha.
\end{equation}
If the coordinates of every $\alpha\in\vartheta(A)$ are even,  the
differential operator $P(D)$ is non-degenerate and it follows from
Theorem~\ref{t:2} that $\|\cdot\|_{W^A_2}$ is equivalent 
to one of the norms in \eqref{asymp[norm-equivalence]}. 
If $\varrho=\varrho(\vartheta(A))$ and $\nu=\nu(\vartheta(A))$, we
again retrieve from Theorem~\ref{t:1} the result proven in 
\cite{DD84}, namely that for $n$ sufficiently large
\begin{equation} 
d_n\big(U^A_2,L_2(\TTd)\big)\asymp n^{-\varrho}
\big(\log n\big)^{\nu\varrho},
\end{equation}
where $U^A_2$ denotes the unit ball in $W^A_2$.

\subsection{Classes of functions induced by a differential operator} 

We give two examples of spaces $\WP$ with non-degenerate
differential operator $P(D)$ for $d=2$. Consider the polynomials
\begin{equation} 
\begin{aligned}
\begin{cases}
P_1\colon x\mapsto
8x_1^4-4x_1^3-3x_1^3x_2-2x_1^2x_2-4x_1x_2+6x_2^2-4x_1-3x_2+13\\
P_2\colon x\mapsto
6x_1^6+x_1^4x_2^2-6x_1^5-x_1^3x_2^2+5x_2^4-4x_2^3+3.
\end{cases}
\end{aligned}
\end{equation}
We have 
\begin{equation} 
\begin{cases}
A_1 &= \{(4,0),(3,0),(2,1),(2,0),(1,1),(0,2),(1,0),(0,1),(0,0)\}\\
\vartheta(A_1)&=\{(4,0),(0,2),(0,0)\}\\
A_2 &= \{(6,0),(4,2),(5,0),(3,2),(0,4),(0,3),(0,0)\}\\
\vartheta(A_2)&=\{(6,0),(4,2),(0,4),(0,0)\}.
\end{cases}
\end{equation}
It is easy to verify that $P_1(D)$ and $P_2(D)$ are non-degenerate
and that \eqref{e:hanoi2014-05-26} holds.
Moreover, $\varrho(\vartheta(A_1))=4/3$, $\nu(\vartheta(A_1))=0$, 
$\varrho(\vartheta(A_2))=8/3$, and $\nu(\vartheta(A_2))=1$.
We derive from Theorem~\ref{t:1} that
\begin{equation} 
d_n\big(U^{[P_1]},L_2(\TT^2)\big) \asymp n^{-4/3}, 
\end{equation}
and
\begin{equation} 
d_n\big(U^{[P_2]},L_2(\TT^2)\big)\asymp n^{-8/3}
\big(\log n\big)^{8/3}.
\end{equation}
Let us give an example of a degenerate differential operator. For 
\begin{equation} 
P_3\colon x\mapsto x_1^4-2x_1^3x_2+x_1^2x_2^2+x_1^2+x_2^2+1, 
\end{equation}
the differential operator $P_3(D)$ is degenerate, although 
$P_3\geq 1$ on $\RR^2$, and $U^{[P_3]}$ is a compact set in 
$L_2(\TT^2)$. Therefore, we cannot compute 
$d_n(U^{[P_3]},L_2(\TT^2))$ by using
Theorem~\ref{t:1}. However, by a direct computation we get 
$\card K(t)\asymp t^{1/2} \log t$. Hence, \eqref{e:7g7} yields
\begin{equation} 
d_n\big(U^{[P_3]},L_2(\TT^2)\big)\asymp n^{-2}\big(\log n\big)^2.
\end{equation}

\subsection{A conjecture}

Suppose that $\UP$ is compact in $L_2(\TTd)$. In view of 
Lemma~\ref{l:7}, this is equivalent to the conditions:
\begin{enumerate}
\item
For every $t\in\RP$, $K(t)$ is finite.
\item 
$\tau>0$.
\end{enumerate}
As mentioned in \eqref{e:7g7}, for every $n\in\NN$ sufficiently
large, if $t(n)\in \RPP$ is the maximal number such that $\card
K(t(n))\leq n$, then 
\begin{equation} 
d_n\big(\UP,L_2(\TTd)\big) \asymp \frac{1}{t(n)}.
\end{equation}
This means that the problem of computing the asymptotic order of
$d_n(\UP,L_2(\TTd))$ is equivalent to the problem of computing 
that of $\card K(t)$ when $t\to\pinf$. Let us 
formulate it as the following conjecture.

\begin{conjecture}
Suppose that, for every $t\in\RP$, $K(t)$ is finite
(the condition $\tau>0$ is not essential). Then there exist 
integers $\alpha$, $\beta$, and $\nu$ such that
$0<\alpha\leq\beta$, $0\leq\nu<d$, and, for $t$ large enough,
\begin{equation}
\card K(t) \asymp t^{\alpha/\beta}\big(\log t\big)^\nu.
\end{equation}
\end{conjecture}

In view of \eqref{e:hanoi}, we know that the conjecture is true 
when $P$ satisfies conditions \eqref{condition[|P(x)|>]} and 
\eqref{e:uj}.

\noindent
{\bfseries Acknowledgment.}
Dinh Dung's research work is funded by Vietnam National Foundation 
for Science and Technology Development (NAFOSTED) under Grant No. 
102.01-2014.02, and a part of it was done when Dinh Dung was working 
as a research professor and Patrick Combettes was visiting at the 
Vietnam Institute for Advanced Study in Mathematics (VIASM). 
Both authors thank  the VIASM for providing fruitful research 
environment and working condition. They also thank the LIA 
CNRS Formath Vietnam for providing travel support.

\end{document}